\documentclass{amsart}
\usepackage[utf8]{inputenc}
\usepackage{colonequals}
\usepackage{color,amssymb,mathrsfs,amsmath,amsfonts,hyperref}

\hypersetup{
    hidelinks=true
} 

\usepackage[capitalise]{cleveref}
\crefformat{equation}{(#2#1#3)}
\crefrangeformat{equation}{(#3#1#4--#5#2#6)}
\crefformat{enumi}{(#2#1#3)}
\crefrangeformat{enumi}{(#3#1#4--#5#2#6)}

\newtheorem{theorem}{Theorem}[section]
\newtheorem*{theorem*}{Theorem}
\newtheorem{lemma}[theorem]{Lemma}
\newtheorem{proposition}[theorem]{Proposition}

\newtheorem{porism}[theorem]{Porism}

\theoremstyle{definition}
\newtheorem{example}[theorem]{Example}
\newtheorem{chunk}[theorem]{}

\theoremstyle{remark}
\newtheorem{remark}[theorem]{Remark}

\newcommand{\Coker}{\operatorname{Coker}}
\newcommand{\injdim}{\operatorname{inj.\!dim}}
\newcommand{\flatdim}{\operatorname{flat\,dim}}
\newcommand{\ZZ}{\mathbb{Z}}
\newcommand{\Spec}{\operatorname{Spec}}
\newcommand{\Ext}{\operatorname{Ext}}
\newcommand{\Tor}{\operatorname{Tor}}
\newcommand{\hh}{\operatorname{H}}
\newcommand{\Hom}{\operatorname{Hom}}
\newcommand{\RHom}{\mathsf{R}\mathrm{Hom}}
\newcommand{\m}{\mathfrak{m}}
\newcommand{\lotimes}{\otimes^{\mathsf{L}}}
\newcommand{\Ker}{\operatorname{Ker}}
\newcommand{\llam}{\mathsf{L}\Lambda}
\newcommand{\rgam}{\mathsf{R}\Gamma}
\newcommand{\fa}{\mathfrak{a}}
\newcommand{\fp}{\mathfrak{p}}
\newcommand{\depth}{\operatorname{depth}}
\newcommand{\width}{\operatorname{width}}

\title{\mbox{Rigidity of Ext and Tor via flat--cotorsion theory}}

\author[L.W.~Christensen]{Lars Winther Christensen}
\address{(L.W.C.) Texas Tech University, Lubbock, TX 79409, U.S.A.}
\email{lars.w.christensen@ttu.edu}

\author[L.~Ferraro]{Luigi Ferraro}
\address{(L.F.) University of Texas Rio Grande Valley, Edinburg, TX 78539, U.S.A.}
\email{luigi.ferraro@utrgv.edu}

\author[P. Thompson]{Peder Thompson}
\address{(P.T.) Mälardalen University, Västerås 72123, Sweden}
\email{peder.thompson@mdu.se}

\thanks{L.W.C.\ was partly supported by Simons Foundation
  collaboration grant 428308.}

\date{19 September 2023}



\begin{document}

\begin{abstract}
  Let $\fp$ be a prime ideal in a commutative noetherian ring $R$ and
  denote by $k(\fp)$ the residue field of the local ring $R_\fp$. We
  prove that if an $R$-module $M$ satisfies $\Ext_R^{n}(k(\fp),M)=0$
  for some $n\geqslant\dim R$, then \mbox{$\Ext_R^i(k(\fp),M)=0$}
  holds for all $i \geqslant n$. This improves a result of
  Christensen, Iyengar, and Marley by lowering the bound on $n$. We
  also improve existing results on Tor-rigidity. This progress is
  driven by the existence of minimal semi-flat-cotorsion replacements
  in the derived category as recently proved by Nakamura~and~Thompson.
\end{abstract}

\maketitle

\section{Introduction}
\noindent
Over a commutative noetherian ring $R$, the injective and flat
dimension of a module can be detected by vanishing of Ext and Tor with
coefficients in residue fields $k(\fp)$ at the prime ideals $\fp$ of
$R$. This drives the interest in \emph{rigidity} properties of Ext and
Tor---here rigidity refers to the phenomenon that vanishing of, say,
$\mathrm{Ext}^n$ implies vanishing of $\mathrm{Ext}^i$ for all
$i\geqslant n$.  Rigidity of Ext and Tor with coefficients in residue
fields was studied by Christensen, Iyengar, and Marley in
\cite{CIM-19}. Here we push the investigation further in two
directions: (1) We eliminate certain asymmetries in the rigidity
statements for Ext/Tor and injective/flat dimension obtained in
\cite{CIM-19}. (2)~We establish and improve results on flat dimension that are
conceptually dual to results on injective dimension already in the
literature, including results obtained in~\cite{CIM-19}.

We work with complexes of modules and our main result, found in
Section~\ref{sec:id},~is:
\begin{theorem*}
  Let $R$ be a commutative noetherian ring and  $M$ an $R$-complex. If for an integer
  $n \geqslant \dim R + \sup{\hh^*(M)}$ one has
  $\Ext_R^{n}(k(\fp),M)=0$ for all prime ideals $\fp$ in $R$, then
  $\injdim_R M < n$ holds.
\end{theorem*}

\noindent This improves \cite[Theorem 5.7]{CIM-19} and aligns
perfectly with \cite[Theorem 4.1]{CIM-19} on flat dimension. By a
result of Christensen and Iyengar \cite[Theorem 1.1]{LWCSBI17}, the
proof reduces to show, for complexes, the statement on Ext-rigidity
made in the Abstract.

In Section~\ref{sec:fd} we prove results on flat dimension of complexes which are dual to already established results on injective dimension. In one of these, we remove a boundedness assumption from \cite[Theorem 4.8]{TNkPTh20}, which is dual to \cite[Corollary~5.9]{CIM-19}:

\begin{theorem*}
Let $R$ be a commutative noetherian ring of finite Krull dimension and $M$ an $R$-complex. If $M$ has finite flat dimesion, then the next equality holds,
  \begin{equation*}
    \flatdim_R M = \sup_{\fp\in\Spec R}\{\depth{R_\fp} - \depth_{R_\fp}\RHom_R(R_\fp,M)\} \,.
  \end{equation*}
\end{theorem*}

%

A novel aspect of our arguments involves the notion of minimal
semi-flat-cotorsion replacements. Recall first that a flat-cotorsion
module is one which is both flat and right $\Ext^1$-orthogonal to flat
modules. A semi-flat complex consisting of flat-cotorsion modules is
called a semi-flat-cotorsion complex. Work of Gillespie \cite{JGl04}
shows that every complex is isomorphic in the derived category to a
semi-flat-cotorsion complex; such a complex is called a
\emph{semi-flat-cotorsion replacement}.

Minimality plays a crucial role in considerations of rigidity and
homological dimensions. A complex $M$ is minimal if every homotopy
equivalence \mbox{$M \to M$} is an isomorphism. Minimal semi-injective
resolutions always exist, and they detect injective dimension; several
proofs in \cite{CIM-19} rely on this. On the other hand, although
semi-flat resolutions always exist, they may not contain a
homotopically equivalent minimal summand that detects the flat
dimension, see \cite[Example 3.9]{LWCPTh19}. Recently, Nakamura and Thompson \cite{TNkPTh20} showed that
every complex over a commutative noetherian ring of finite Krull
dimension has a minimal semi-flat-cotorsion replacement and that such
a complex detects the flat dimension. 

In this paper, we also take the opportunity to clarify a couple of statements in
\cite{CIM-19}; see Remarks \ref{rmk:1} and \ref{rmk:2}.
\begin{equation*}
  \ast \ \ \ast \ \ \ast
\end{equation*}
In this paper $R$ is a commutative noetherian ring. We widely adopt
the notation used in \cite{CIM-19}; in particular, we use both
homological and cohomological notation.

\begin{chunk}
  Let $M$ be an $R$-complex.

  For an integer $s$ we denote by $\Sigma^sM$ the complex with the
  module $M_{i-s}$ in degree $i$ and differential
  $d^{\Sigma^sM}=(-1)^sd^M$.

  Set $\inf \hh_*(M)\colonequals \inf\{i \mid \hh_i(M)\not=0\}$ and
  $\inf \hh^*(M)\colonequals \inf\{i\mid \hh^i(M)\not=0\}$, and define
  $\sup\hh_*(M)$ and $\sup\hh^*(M)$ similarly.

  In case $R$ is local with maximal ideal $\m$, the right derived
  $\m$-torsion functor is denoted $\rgam_\m$ and $\llam^{\m}$ is the
  left derived $\m$-completion functor. The corresponding local
  (co)homology modules are denoted $\hh_\m^*(M)$ and $\hh^\m_*(M)$. As
  always,
  \begin{equation*}
    \depth_RM = \inf \Ext^*_R(k,M) \quad\text{and}\quad
    \width_RM = \inf \Tor_*^R(k,M) 
  \end{equation*}
  where $k$ denotes the residue field $R/\m$.
\end{chunk}

We recall from \cite{TNkPTh20} some properties of minimal
semi-flat-cotorsion complexes that will be used frequently.

\begin{chunk}
  \label{chunk}
  Assume that $R$ has finite Krull dimension and let $M$ be an
  $R$-complex. By \cite[Theorem 3.4]{TNkPTh20} there exists a minimal
  semi-flat-cotorsion complex $F$ isomorphic to $M$ in the derived
  category: a minimal semi-flat-cotorsion replacement of $M$.
    
  If $M$ has finite flat dimension, then it follows from \cite[Lemma
  4.1]{TNkPTh20} that $F_i = 0$ holds for $i > \flatdim_RM$.
    
  If $\hh_i(M) = 0$ holds for $i \ll 0$, then $F_i = 0$ for $i \ll 0$
  by \cite[Lemma~4.1]{TNkPTh20}.
    
  For every prime ideal $\fp$ in $R$ the complex $\Hom_R(R_\fp,F)$ is
  isomorphic in the derived category to $\RHom_R(R_\fp,M)$, see
  \cite[(A.1)]{TNkPTh20}. Further, it consists by \cite[Lemma
  2.2]{PTh19} of flat-cotorsion $R$-modules, so if $F_i = 0$ for
  $i \ll 0$, then $\Hom_R(R_\fp,F)$ is a semi-flat-cotorsion
  replacement of $\RHom_R(R_\fp,M)$ over both $R$ and $R_\fp$.
\end{chunk}

\section{Rigidity of Tor}
\label{sec:Tor}
\noindent
The main result of this section, Theorem \ref{thm:supTor} below,
removes a boundedness condition on $\hh_*(M)$ from
\cite[Proposition 3.3(1)]{CIM-19} and aligns perfectly---see
Remark~\ref{rmk:0} below---with \cite[Proposition 3.2]{CIM-19} on
rigidity of Ext.

\begin{lemma}
  \label{lem:TenskLam}
  Let $(R,\m,k)$ be a local ring and $F$ a minimal semi-flat-cotorsion
  $R$-complex. There is an isomorphism of $R$-complexes
  $k\otimes_RF\cong k\otimes_R\Lambda^\m(F)$, and both complexes have
  zero differential.
\end{lemma}

\begin{proof}
  The canonical map $F\to \Lambda^\m(F)$ induces per \cite[Theorem
  2.2.2]{PScAMS} an isomorphism of complexes
  $k\otimes_R F\to k\otimes_R \Lambda^\m(F)$.  Minimality of $F$
  implies per \cite[Theorem 2.3]{TNkPTh20} that $k\otimes_R F$, and
  hence also $k\otimes_R \Lambda^\m(F)$, has zero differential.
\end{proof}

\begin{theorem}
  \label{thm:supTor}
  Let $(R,\m,k)$ be a local ring and $M$ an $R$-complex. If one has
  $\Tor^R_{n+1}(k,M)=0$ for an integer $n \geqslant\sup\hh^\m_*(M)$,
  then $\Tor^R_i(k,M)=0$ holds for all $i > n$ and the next equality
  holds,
  \begin{equation*}
    \sup{\Tor^R_*(k,M)} = \depth R - \depth_R M \,.
  \end{equation*}
\end{theorem}

\begin{proof}
  If the complex $k\lotimes_RM$ is acyclic, then one has
  $\sup\Tor^R_*(k,M)=-\infty$, so the assertion of the theorem is
  trivial, and so is the equality since $\depth_RM=\infty$ holds in
  this case, see \cite[Definitions 2.3 and 4.3]{HBFSIn03}. Hence we may
  assume that $k\lotimes_RM$ is not acyclic. From \cite[(2.6)]{CIM-19}
  it follows that $\llam^\m(M)$ is not acyclic. If
  $\sup\hh^\m_*(M)=\infty$ holds, then the statement is vacuously
  true. Thus, we may assume that $\sup\hh^\m_*(M)<\infty$ holds and
  set $s \colonequals \sup\hh^\m_*(M)$.

  Let $F$ be a minimal semi-flat-cotorsion replacement of $M$, see
  \ref{chunk}, and set $P = \Lambda^\m(F)$. As $F$ is semi-flat, one
  has $\llam^\m(M)=P$, see for example \cite[Proposition 3.6]{PSY-14},
  and \cref{lem:TenskLam} yields the isomorphism
  $k\lotimes_RM \simeq k\otimes_RP$ in the derived category. For every
  $i\in\mathbb{Z}$ the $\m$-complete module $P_i = \Lambda^\m(F_i)$ is
  flat-cotorsion, see \cite[Lemma 2.2]{PTh19}. The complex
  $\Sigma^s(P_{\geqslant s})$ is a semi-flat-cotorsion replacement of
  the module $C = \Coker({P_{s+1}\rightarrow P_s})$.  For every
  integer $i > s$ there are isomorphisms
  \begin{equation*}
    \Tor_i^R(k,M) \cong \hh_i(k\otimes_RP) \cong
    \hh_i(k\otimes_RP_{\geqslant s}) \cong \Tor_{i-s}^R(k,C) \,.
  \end{equation*}
  Let $n\geqslant s$ and assume that $\Tor^R_{n+1}(k,M)=0$ holds. By
  the isomorphisms above one has $\Tor_{n+1-s}^R(k,C)=0$. The complex
  $P$, and hence the truncated complex $P_{\geqslant s}$, is
  $\m$-complete, so the module $C$ is derived $\m$-complete. It now
  follows from \cite[Lemma 2.1]{CIM-19} that the module $C$ has flat
  dimension at most $n-s$; in particular,
  $\Tor_i^R(k,M) = \Tor^R_{i-s}(k,C)=0$ holds for all
  $i\geqslant n+1$. This proves the first claim.
  
  To prove the asserted equality, let $E$ be the injective envelope of
  $k$. Adjunction yields
  \begin{equation*}
      \Hom_R(\Tor_i^R(k,M),E) \cong \Ext^i_R(k,\Hom_R(M,E))\,,
  \end{equation*}
  so
  \mbox{$\Ext^i_R(k,\Hom_R(M,E)) = 0$} holds for $i \gg 0$. Now
  faithful injectivity of $E$, together with \cite[Proposition 3.2]{CIM-19}, yields
  \begin{align*}
    \sup{\Tor^R_*(k,M)} &= \sup{\Ext_R^*(k,\Hom_R(M,E))} \\
                        & = \depth R- \width_R \Hom_R(M,E) \\
                        & = \depth R - \depth_R M \,,   
  \end{align*}
  where the last equality is standard, see for example
  \cite[Proposition 4.4]{HBFSIn03}.
\end{proof}

\begin{remark}
  \label{rmk:0}
  The bound on $n$ in \cite[Proposition 3.2]{CIM-19} appears to be $1$
  lower than the bound in Theorem~\ref{thm:supTor}, but as noted in
  the opening paragraph of \cite[Section~3]{CIM-19} the lower bound
  cannot possibly be attained.  We show below that one could
  similarly lower the bound in Theorem~\ref{thm:supTor} by $1$ as $\Tor_s^R(k,M) \ne 0$ holds for $s = \sup\hh^\m_*(M)$.

  Let $M$ be an $R$-complex and $F$ a minimal semi-flat-cotorsion
  replacement of $M$. Set $n=\sup \hh^\m_*(M)$.  In degree $n$, the
  complex $\Lambda^\m(F)$ is nonzero and Lemma~\ref{lem:TenskLam}
  yields $\Tor^R_n(k,M) = k\otimes_R \Lambda^\m(F)_n$, which is
  nonzero, see \cite[Observation 2.1.2]{PScAMS}.
\end{remark}

\begin{remark}
  \label{rmk:1}
  The conclusion of \cite[Lemma 2.1]{CIM-19} states that
  $\Tor_n^R(-,M)=0$ holds if $M$ is a derived $\fa$-complete complex
  with $\inf\hh_*(M) > -\infty$ and $\Tor_n^R(R/\fp,M)=0$ holds for
  all prime ideals $\fp$ that contain $\fa$. We notice here that the proof
  in \cite{CIM-19} only demonstrates this for $\Tor_n^R(-,M)$ as a
  functor on the category of $R$-modules (not $R$-complexes). This is
  sufficient for the purposes of its use in both \cite{CIM-19} and the
  proof above. To see that the conclusion fails for Tor as a functor
  on complexes, let $(R,\m,k)$ be a complete local ring and notice that though $\Tor_n^R(k,R) = 0$ holds for every
  $n \geqslant 1$ one has $\Tor_n^R(\Sigma^nk,R) \cong k$.
\end{remark}

\section{Injective dimension}
\label{sec:id}
\noindent The next result improves the bound on $n$ in \cite[Proposition 5.4 and Theorem 5.7]{CIM-19}.

\begin{theorem}
  \label{thm:main_Ext}
  Let $R$ be a commutative noetherian ring and $M$ an $R$-complex. If
  for an integer $n \geqslant \dim R + \sup{\hh^*(M)}$ one has
  $\Ext_R^{n}(k(\fp),M)=0$ for all prime ideals $\fp$ in $R$, then
  $\injdim_R M < n$ holds.
\end{theorem}

\begin{proof}
  We may assume that $R$ has finite Krull dimension and that $M$ is
  not acyclic. We may also assume that $\hh^i(M) = 0$ holds for
  $i \gg 0$, otherwise the statement is vacuous. For every prime ideal
  $\fp$ in $R$ one has
  \begin{equation*}
    0 = \Ext^{n}_R(k(\fp),M)
    \cong \Ext^{n}_{R_\fp}(k(\fp),\RHom_R(R_\fp,M))
  \end{equation*}
  by Hom--tensor adjunction in the derived category.  It suffices, by
  \cite[Theorem 1.1]{LWCSBI17} and \cite[Proposition 3.2]{CIM-19}, to
  show that
  $\dim R + \sup \hh^*(M) \geqslant \sup{\hh^*_{\fp_\fp}\RHom_R(R_\fp,M)}$ holds.
  For every $R_\fp$-complex $X$ there is an isomorphism $\rgam_{\fp_\fp}X \simeq \rgam_\fp X$ in the derived category over $R_\fp$; this follows for example from \cite[Lemma (3.2.3)]{AJL-97} and explains the first and last isomorphisms in the next display. The second isomorphism holds by \cite[Corollary
  (5.1.1)]{AJL-97} while the third comes from \cite[Proposition 8.3]{BKI-12}.
  \begin{equation}
    \label{ast2}
    \begin{aligned}
      \rgam_{\fp_\fp}\RHom_R(R_\fp,M) & \simeq \rgam_{\fp}\RHom_R(R_\fp,M)\\ &\simeq\rgam_{\fp}\llam^{\fp}\RHom_R(R_\fp,M) \\
      & \simeq \rgam_{\fp}\RHom_R(R_\fp,\llam^\fp M)\\
      &\simeq\rgam_{\fp_\fp}\RHom_R(R_\fp,\llam^\fp M)\,.
    \end{aligned}
  \end{equation}
   Let $F$ be a minimal semi-flat-cotorsion
  replacement of $M$. In cohomological notation one has $F^i=0$ for $i \gg 0$, see \ref{chunk},
  so the complex $\Lambda^\fp F \simeq \llam^\fp M$ is again a
  semi-flat-cotorsion $R$-complex, see \cite[Lemma 2.2]{PTh19}. In the
  derived category, the complex $P = \Hom_R(R_\fp,\Lambda^\fp F)$ is
  now isomorphic to $\RHom_R(R_\fp,\llam^\fp M)$. Further \cite[Lemma 4.1]{TNkPTh20} yields $P^i = 0$ for
  $i > \sup{\hh^*(M)} + \dim R/\fp$, which explains the second
  inequality in the computation below. The first equality holds by
  \eqref{ast2} and the first inequality holds by \cite[(2.7)]{CIM-19},
  \begin{align*}
    \sup{\hh^*_{\fp_\fp}(\RHom_R(R_\fp,M))} 
    & = \sup{\hh^*(\rgam_{\fp_\fp}\RHom_R(R_\fp,\llam^\fp M))}  \\
    & \leqslant \dim R_\fp + \sup{\hh^*(\RHom_R(R_\fp,\llam^\fp M))} \\
    & = \dim R_\fp + \sup{\hh^*(P)} \\
    & \leqslant \dim R_\fp + \dim R/\fp + \sup{\hh^*(M)} \\
    & \leqslant \dim R + \sup{\hh^*(M)} \,. \qedhere
  \end{align*}
\end{proof}

From the proof above one easily extracts the following rigidity
result; for modules it was stated in the Abstract.

\begin{porism}
  \label{por:por}
  Let $\fp$ be a prime ideal in $R$ and $M$ an $R$-complex. If for an
  integer $n \geqslant \dim R + \sup{\hh^*(M)}$ one has
  $\Ext_R^{n}(k(\fp),M)=0$, then $\Ext_R^{i}(k(\fp),M)=0$ holds for
  all $i \geqslant n$.
\end{porism}

The bound on $n$ in Theorem \ref{thm:main_Ext} and Porism
\ref{por:por} is sharp: Let $(R,\m,k)$ be a Cohen-Macaulay local ring
that is not Gorenstein. One has
\begin{equation*}
  \inf \Ext_R^*(k,R) = \depth R = \dim R \quad\text{but}\quad \sup \Ext_R^*(k,R) =\injdim_R R= \infty \,.
\end{equation*}
This also shows that the bound on $n$ in the next proposition is sharp.
This  statement is parallel to the first part of \cite[Theorem
4.1]{CIM-19} and could have been made in~\cite{CIM-19}.

\begin{proposition}
  \label{prp:supExt}
  Let $R$ be a commutative noetherian ring and $M$ an $R$-complex. If
  for a prime ideal $\fp$ and an integer
  $n \geqslant \dim R_\fp + \sup{\hh^*(\RHom_R(R_\fp,M))}$ one has
  $\Ext_R^{n}(k(\fp),M)=0$ then
  \begin{equation*}
    \sup\Ext^*_R(k(\fp),M) = \depth R_\fp - \width_{R_\fp} \RHom_R(R_\fp,M) < n \,.
  \end{equation*}
\end{proposition}

\begin{proof}
  Hom--tensor adjunction in the derived category yields for every
  integer $n$ an isomorphism
  \begin{equation*}
    \Ext_R^n(k(\fp),M) \cong \Ext_{R_\fp}^n(k(\fp),\RHom(R_\fp,M)) \,.
  \end{equation*}
  Thus, the assertions follow immediately from \cite[Proposition
  3.2]{CIM-19} since one has
  $n \geqslant \dim R_\fp + \sup{\hh^*(\RHom_R(R_\fp,M))} \geqslant
  \sup\hh_{\fp_\fp}^*(\RHom(R_\fp,M))$; see \cite[(2.7)]{CIM-19}.
\end{proof}

\section{Flat dimension}
\label{sec:fd}
\noindent
In this section we prove three statements that are dual to statements
about injective dimension in the literature. Our proofs rely on the
existence and structure of minimal semi-flat-cotorsion replacements,
hence the assumption that the ring has finite Krull dimension. The
first result below is a counterpart to \cite[Theorem~5.1]{CIM-19}, and could have been stated even in \cite{TNkPTh20}.

\begin{theorem}
  \label{thm:fdRigid}
  Let $R$ be a commutative noetherian ring of finite Krull dimension
  and $M$ an $R$-complex with $\inf\hh_*(M)>-\infty$. If for an
  integer $n\geqslant\sup\hh_*(M)$ one has
  \begin{equation*}
    \Tor_{n+1}^{R_\fp}(k(\fp),\RHom_R(R_\fp,M))=0\;\text{ for every
      prime ideal }\fp\text{ in }R \,,
  \end{equation*}
  then the flat dimension of $M$ is at most $n$.
\end{theorem}

\begin{proof}
  Let $F$ be a minimal semi-flat-cotorsion replacement of $M$; the assumption $\inf\hh_*(M)>-\infty$ guarantees that $F_i=0$ holds for
  $i\ll0$, see \ref{chunk}. Per \cite[Remark 4.5]{TNkPTh20}, for every integer $i$ one has
  \begin{equation*}
    F_i = \prod_{\fp\in\Spec R}\Lambda^\fp\left(R_\fp^{(B^\fp_i)}\right) \,,
  \end{equation*}
  where the cardinality of $B_i^\fp$ 
  is the $\kappa(\fp)$-dimension of $\Tor^{R_\fp}_{i}(k(\fp),\RHom_R(R_\fp,M))$. Thus it follows from the assumption that $|B^\fp_{n+1}|=0$ holds for all primes $\fp$. As $n\geqslant\sup\hh_*(M)$ holds, the flat dimension of $M$ is at most $n$.
%
\end{proof}

The boundedness condition in the next result, which is dual to
\cite[Proposition 5.3.I]{LLAHBF91}, is necessary, without it the flat
dimension may grow under colocalization; see Example~\ref{exa:1}.

\begin{proposition}
  \label{prop:FlatDimTor}
  Let $R$ be a commutative noetherian ring of finite Krull dimension
  and $M$ an $R$-complex with $\inf{\hh_*(M)} > -\infty$. The next
  equalities hold
  \begin{align*}
    \flatdim_R M &= \sup{ \{ i \in \ZZ \mid \Tor_i^{R_\fp}(k(\fp),\RHom(R_\fp,M))
                   \ne 0 \text{ for some } \fp\in\Spec R \}} \\
                 & = \sup_{\fp\in\Spec R}{\{\flatdim_{R_\fp}\RHom_R(R_\fp,M) \}} \,.
  \end{align*}
\end{proposition}

\begin{proof}
  If $M$ is acyclic, all three quantities equal $-\infty$, and so we may
  assume $\hh_*(M)$ is nonzero.  Set $s \colonequals \sup\hh_*(M)$,
  one has $\Tor_s^{R_\fp}(k(\fp),\RHom(R_\fp,M)) \ne 0$ for some prime
  ideal $\fp$ in $R$, see \cite[Remark 4.5]{TNkPTh20}. The inequality
  \begin{equation*}
    \flatdim_R M \leqslant \sup{ \{ i \in \ZZ \mid \Tor_i^{R_\fp}(k(\fp),\RHom(R_\fp,M))
      \ne 0 \text{ for some } \fp\in\Spec R \}}
  \end{equation*}
  now follows immediately from \cref{thm:fdRigid}. To verify the opposite
  inequality, assume that $f \colonequals \flatdim_R M$ is finite and
  let $F$ be a minimal semi-flat-cotorsion replacement of $M$. As
  $\inf\hh_*(M)>-\infty$ holds, one has $F_i=0$ for $i > f$ and
  $i\ll0$, see \ref{chunk}. For every prime ideal $\fp$ in $R$, the
  $R_\fp$-complex $\Hom_R(R_\fp,F)$ is semi-flat-cotorsion and
  isomorphic to $\RHom_R(R_\fp,M)$ in the derived category over
  $R_\fp$. Since $\Hom_R(R_\fp,F)_i = \Hom_R(R_\fp,F_i) = 0$ holds for
  $i > f$ one has
  \begin{align*}
    \sup{\{i\in\ZZ \mid
    \Tor_i^{R_\fp}(k(\fp),\RHom(R_\fp,M)) \ne 0\}} %
    & \leqslant \flatdim_{R_\fp}\RHom_R(R_\fp,M) \\
    & \leqslant \flatdim_R M \,. \qedhere
  \end{align*}
\end{proof}

For complexes with bounded homology the next result was proved in
\cite[Theorem 4.8]{TNkPTh20}; it compares to \cite[Proposition 5.2 and
Corollary 5.9]{CIM-19}, and the proof is modeled on the proof of
\cite[Proposition 5.2]{CIM-19}.

\begin{theorem}
  Let $R$ be a commutative noetherian ring of finite Krull dimension
  and $M$ an $R$-complex. If $M$ has finite flat dimension, then
  \begin{equation*}
    \flatdim_R M = \sup_{\fp\in\Spec R}\{\depth{R_\fp} - \depth_{R_\fp}\RHom_R(R_\fp,M)\} \,.
  \end{equation*}
\end{theorem}

\begin{proof}
  The equality is trivial if $M$ is acyclic, so assume that it is not
  and set $f \colonequals \flatdim_R M$. Let $F$ be a semi-flat
  replacement of $M$ with $F_i =0$ for $i > f$. To prove the asserted
  equality it suffices to show that the inequality
  \begin{equation}\label{ast}
    \flatdim_R F \geqslant \depth{R_\fp} - \depth_{R_\fp}\RHom_R(R_\fp,F)
  \end{equation}
  holds for every prime $\fp$ with equality for some $\fp$. For every
  $n \leqslant f$ there is an exact sequence of complexes of flat
  $R$-modules
  \begin{equation*}
    0 \longrightarrow F_{\leqslant n-1} \longrightarrow F \longrightarrow F_{\geqslant n} \longrightarrow 0 \,.
  \end{equation*}
  The complex $F_{\geqslant n}$ is a bounded complex of flat modules,
  so it is semi-flat and hence so is $F_{\leqslant n-1}$; see for
  example \cite[6.1]{LWCHHl15}. Evidently one has
  \begin{equation*}
    \flatdim_R F_{\leqslant n-1} \leqslant n-1 \quad\text{and}\quad \flatdim_R F_{\geqslant n} = f\,.
  \end{equation*}
  Now \cref{prop:FlatDimTor} and \cite[(2.3)]{CIM-19} conspire to
  yield
  \begin{equation}
    \label{sharp}
    f = \sup_{\fp\in\Spec R}\{\depth R_\fp - \depth_{R_\fp}\RHom_R(R_\fp,F_{\geqslant n})\} \,.
  \end{equation}
  To prove \eqref{ast} we fix a prime ideal $\fp$. Without loss of
  generality we can assume that $\depth_{R_\fp}\RHom_R(R_\fp,F)$ is
  finite, set $d = -\depth_{R_\fp}\RHom_R(R_\fp,F)$. Now one has
  \begin{equation*}
    -d \geqslant \inf{\hh^*(\RHom_R(R_\fp,F)}) \geqslant \inf{\hh^*(F)} = -\sup{\hh_*(F)} \geqslant -f
  \end{equation*} and, thus, $d \leqslant f$. As one now has
  \begin{align*}
    \depth_{R_\fp}\RHom_R(R_\fp,F_{\leqslant d-1}) %
    &\geqslant \inf \hh^*(\RHom_R(R_\fp,F_{\leqslant d-1})) \\
    &\geqslant \inf \hh^*(F_{\leqslant d-1}) \\
    &\geqslant 1-d\\
    &= 1 + \depth_{R_\fp}\RHom_R(R_\fp,F)\,,
  \end{align*}
  the depth lemma yields
  \begin{equation*}
    \depth_{R_\fp}\RHom_R(R_\fp,F) = \depth_{R_\fp}\RHom_R(R_\fp,F_{\geqslant d}) \,.    
  \end{equation*} 
  The inequality \eqref{ast} follows from \eqref{sharp} by taking
  $n=d$.
    
  Now choose by \eqref{sharp} a prime ideal $\fp$ such that
  \begin{equation}
    \label{f}
    f = \depth R_\fp - \depth_{R_\fp}\RHom_R(R_\fp,F_{\geqslant f-1})
  \end{equation} 
  holds. The inequality \eqref{ast} holds for every complex of finite
  flat dimension; applied to the truncated complex $F_{\leqslant f-2}$
  it yields
  \begin{equation}
    \label{f-2}
    f-2 \geqslant \flatdim_R F_{\leqslant f-2} \geqslant \depth R_\fp - \depth_{R_\fp}\RHom_R(R_\fp,F_{\leqslant f-2}) \,.
  \end{equation}
  Elimination of $\depth R_\fp$ between \eqref{f} and \eqref{f-2}
  yields
  \begin{equation*}
    \depth_{R_\fp}\RHom_R(R_\fp,F_{\geqslant f-1}) + 2  \leqslant  \depth_{R_\fp}\RHom_R(R_\fp,F_{\leqslant f-2})\,.
  \end{equation*}
  Now apply the depth lemma to the triangle
  \begin{equation*}
    \RHom_R(R_\fp,F_{\leqslant f-2}) \longrightarrow \RHom_R(R_\fp,F) \longrightarrow 
    \RHom_R(R_\fp,F_{\geqslant f-1}) \longrightarrow
  \end{equation*}
  to get
  $\depth_{R_\fp}\RHom_R(R_\fp,F)=\depth_{R_\fp}\RHom_R(R_\fp,F_{\geqslant
    f-1})$; substituting this into \eqref{f} yields the desired
  result.
\end{proof}

\begin{remark}
  For an $R$-complex $M$ of finite flat dimension, the equality
  \begin{equation*}
    \flatdim_R M = \sup_{\fp\in\Spec R}\{\depth{R_\fp} - \depth_{R_\fp}M_\fp\}
  \end{equation*}
  holds, see \cite[(4.3)]{CIM-19}. Echoing \cite[Remark 5.10]{CIM-19}
  we remark that we do not know how the numbers $\depth_{R_\fp}M_\fp$
  and $\depth_{R_\fp}\RHom_R(R_\fp,M)$ compare: If $(R,\m,k)$ is local
  of positive Krull dimension and $E$ is the injective envelope of $k$, then one has
  $\depth_{R_\fp}\RHom_R(R_\fp,E) = 0$ for every prime
  ideal $\fp$ by the isomorphisms
  \begin{equation*}
    \RHom_{R_\fp}(k(\fp),\RHom_R(R_\fp,E)) \simeq
    \RHom_R(k(\fp),E) \simeq \Hom_R(k(\fp),E) \ne 0 \,.
  \end{equation*}
  On the other hand, for $\fp \not = \m$ the module
  $E_\fp$ is zero and hence of infinite depth.
\end{remark}

\section{Examples}
\noindent
We close with a series of examples to show that boundedness
conditions, such as the one in Proposition \ref{prop:FlatDimTor}, are
necessary in certain statements about semi-flat complexes. In
particular, neither colocalization nor completion need preserve
finiteness of flat dimension. The examples build on \cite[Example
5.11]{TNkPTh20}.

\begin{example}
  \label{exa:1}
  Let $k$ be a field and consider the local ring $R=k[[x,y]]/(x^2)$;
  it has only two prime ideals: $\fp=(x)$ and $\m=(x,y)$.  By
  \cite[Example 5.11]{TNkPTh20} there exists a minimal
  semi-flat-cotorsion $R$-complex $Y$, called $Y_P$ in \cite{TNkPTh20}, such that $\Hom_R(R_\fp,Y)$
  is not semi-flat.

  Set $F=(Y)_{\leqslant 0}$ and $F'=(Y)_{\geqslant 1}$, the hard
  truncations of $Y$ above at 0 and below at 1, respectively. There
  is an exact sequence $0 \to F \to Y \to F' \to 0$ which is
  degreewise split.  Application of $\Hom_R(R_\fp,-)$ now yields the
  exact sequence
  \begin{equation*}
    0 \longrightarrow \Hom_R(R_\fp,F)  \longrightarrow \Hom_R(R_\fp,Y)
    \longrightarrow \Hom_R(R_\fp,F') \longrightarrow 0 \,.
  \end{equation*}
  Since the complex $\Hom_R(R_\fp,F')$ consists of flat modules and
  $\Hom_R(R_\fp,F')_i=0$ holds for $i\ll0$, it is semi-flat. Since
  $\Hom_R(R_\fp,Y)$ is not semi-flat, neither is the complex
  $\Hom_R(R_\fp,F)$.

  It follows from \cite[Theorem 2.3]{TNkPTh20} that $F$ is a minimal
  semi-flat-cotorsion complex; as $\hh_0(F) \ne 0$ it has flat
  dimension $0$. We claim, however, that $\Hom_R(R_\fp,F)$ has
  infinite flat dimension. Suppose, to the contrary, that the complex
  $\Hom_R(R_\fp,F)$ has finite flat dimension. Choose a semi-flat
  resolution $P \to \Hom_R(R_\fp,F)$ with $P_i=0$ for $i\gg0$ and let
  $C$ be its mapping cone. Evidently, $C$ is an acyclic complex of
  flat $R$-modules and $C_i=0$ holds for $i\gg0$. We argue that $C$ is
  pure-acyclic; that is, all its cycle modules are flat: For every
  integer $i$, let $Z_i=\Ker(C_i\to C_{i-1})$ be the cycle module in
  degree $i$. There is an exact sequence
  $0 \to Z_i \to C_i \to Z_{i-1} \to 0$. The ring $R$ is Gorenstein of
  Krull dimension 1, so by a result of Bass \cite[Corollary
  5.6]{HBs62} the finitistic flat dimension of $R$ is $1$. As $C_i$ is
  flat, it follows that $Z_i$ is flat. As $i$ was arbitrary, this
  shows that $C$ is pure-acyclic. It now follows from \cite[Theorem
  7.3]{LWCHHl15} that $C$ is semi-flat. As $C$ fits in a short exact
  sequence with the complexes $P$ and $\Hom_R(R_\fp,F)$, of which the
  latter is not semi-flat, this is a contradiction. It follows that
  $\Hom_R(R_\fp,F)$ has infinite flat dimension.
\end{example}

One can draw the same conclusion as in the previous example about the
$\m$-completion of $F$:
  
\begin{example}
  \label{exa:3}
  Let $R$ and $F$ be as in Example~\ref{exa:1}.  By
  \cite[(1.17)]{TNkPTh20} there is an exact sequence
  \begin{equation*}
    0 \longrightarrow \Hom_R(R_\fp,F) \longrightarrow F \longrightarrow
    \Lambda^{\m}(F) \longrightarrow 0 \,.  
  \end{equation*}
  As $F$ is semi-flat of finite flat dimension, equal to 0, but
  $\Hom_R(R_\fp,F)$ is not semi-flat and does not have finite flat
  dimension, it follows that the complex $\Lambda^{\m}(F)$ is not
  semi-flat and does not have finite flat dimension.
\end{example}

The examples above are dual to the examples in \cite[Section
6]{CIM-19}, and we take this opportunity to clarify one of the
statements made there.

\begin{remark}
  \label{rmk:2}
  Let $(R,\m,k)$ be local. It is stated in \cite[Remark 6.2]{CIM-19}
  that the support of the product $\operatorname{E}_R(k)^\mathbb{N}$
  is all of $\Spec{R}$, and that is correct though the argument
  provided in \cite{CIM-19} is too brief to be accurate. Here is a
  complete argument:
    
  Set $E = \operatorname{E}_R(k)$ and for every $n \in \mathbb{N}$ let
  $E_n$ denote the submodule
  \begin{equation*}
    (0:_E \m^n) \cong \Hom_R(R/\m^n,E) \,.
  \end{equation*} 
  It is the injective envelope of the artinian ring $R/\m^n$, so one
  has $(0:_RE_n) = \m^n$. Indeed, for $x$ in $R$ the isomorphism
  $R/\m^n \cong \Hom_{R/\m^n}(E_n,E_n) \cong \Hom_R(E_n,E_n)$
  identifies the homothety $E_n \xrightarrow{x} E_n$ with the coset
  $x + \m^n$ in $R/\m^n$. Each submodule $E_n$ has finite length; in
  particular, it is generated by elements
  $e_{n,1},\ldots,e_{n,m_n}$. Now let $e$ be the family of all these
  generators in the countable product
  $\prod_{n\in\mathbb{N}}\prod_{i=1}^{m_n}E$. It follows from Krull's
  intersection theorem that the homomorphism $R \to E^\mathbb{N}$
  given by $1 \mapsto e$ is injective, since an element in the kernel
  annihilates $E_n$ for every $n\in\mathbb{N}$ and hence belongs to
  the intersection $\bigcap_{n\in\mathbb{N}}\m^n$. As localization is
  exact, $R_\fp$ is now a non-zero submodule of $(E^\mathbb{N})_\fp$
  for every prime ideal $\fp$ in $R$.
\end{remark}

 Let $R$ be a commutative noetherian ring, $P$ a projective
  $R$-module, and $F$ a semi-flat $R$-complex.  As products of flat
  $R$-modules are flat, $\Hom_R(P,F)$ is a complex of flat
  $R$-modules. If $F_i=0$ holds for $i\ll0$, then the complex
  $\Hom_R(P,F)$ satisfies the same boundedness condition, whence it is
  semi-flat. Without this boundedness condition the conclusion may fail.

\begin{example}
  Let $R$ be the ring and $F$ the semi-flat $R$-complex from
  Example~\ref{exa:1}. There exist projective $R$-modules $P$ such
  that $\Hom_R(P,F)$ is not semi-flat: As $R$ is Gorenstein of Krull
  dimension 1, the finitistic projective dimension of $R$ is $1$ by
  \cite[Corollary 5.6]{HBs62}. It follows that there is an exact
  sequence $0 \to P_1 \to P_0 \to R_\fp \to 0$ where $P_1$ and $P_0$
  are projective $R$-modules. As $F$ is a complex of flat-cotorsion
  modules, in particular modules that are $\mathrm{Ext}^1$-orthogonal
  to $R_\fp$, it yields an exact sequence
  \begin{equation*}
    0 \longrightarrow\Hom_R(R_\fp,F) \longrightarrow \Hom_R(P_0,F) \longrightarrow \Hom_R(P_1,F) 
    \longrightarrow 0 \,.
  \end{equation*}
  Assume towards a contradiction that $\Hom_R(P_1,F)$ and
  $\Hom_R(P_0,F)$ are both semi-flat $R$-complexes. As $F_i = 0$ holds
  for $i>0$, it follows that both complexes have finite flat
  dimension, at most $0$, and hence so has $\Hom_R(R_\fp,F)$. This
  contradicts the conclusion in Example~\ref{exa:1} that
  $\Hom_R(R_\fp,F)$ has infinite flat dimension.
\end{example}

For the ring $R=k[[x,y]]/(x^2)$ with $\fp=(x)$ from
Example \ref{exa:1} one has $R_\fp = R_y$, so it follows from
\cite[Example 1.6]{TNkPTh20} that the modules $P_0$ and $P_1$
can be chosen as countable direct sums of copies of
  $R$. Thus, $F$ is an example of a semi-flat complex such that the
product $F^\mathbb{N}$ is not semi-flat. Compare this to the fact that
for a semi-injective complex $I$ the coproduct $I^{(\mathbb{N})}$ need
not be semi-injective; see Iacob and Iyengar \cite[Theorem
2.8]{AIcSBI09}.

\section*{Acknowledgments}
\noindent
It is our pleasure to thank Srikanth Iyengar and Thomas Marley for discussions related to this work. We also thank the anonymous referee for their feedback, including the observation about the modules $P_0$ and $P_1$ recounted right above.

\def\soft#1{\leavevmode\setbox0=\hbox{h}\dimen7=\ht0\advance \dimen7
  by-1ex\relax\if t#1\relax\rlap{\raise.6\dimen7
  \hbox{\kern.3ex\char'47}}#1\relax\else\if T#1\relax
  \rlap{\raise.5\dimen7\hbox{\kern1.3ex\char'47}}#1\relax \else\if
  d#1\relax\rlap{\raise.5\dimen7\hbox{\kern.9ex \char'47}}#1\relax\else\if
  D#1\relax\rlap{\raise.5\dimen7 \hbox{\kern1.4ex\char'47}}#1\relax\else\if
  l#1\relax \rlap{\raise.5\dimen7\hbox{\kern.4ex\char'47}}#1\relax \else\if
  L#1\relax\rlap{\raise.5\dimen7\hbox{\kern.7ex
  \char'47}}#1\relax\else\message{accent \string\soft \space #1 not
  defined!}#1\relax\fi\fi\fi\fi\fi\fi}
  \providecommand{\MR}[1]{\mbox{\href{http://www.ams.org/mathscinet-getitem?mr=#1}{#1}}}
  \renewcommand{\MR}[1]{\mbox{\href{http://www.ams.org/mathscinet-getitem?mr=#1}{#1}}}
  \providecommand{\arxiv}[2][AC]{\mbox{\href{http://arxiv.org/abs/#2}{\sf
  arXiv:#2 [math.#1]}}} \def\cprime{$'$}
\providecommand{\bysame}{\leavevmode\hbox to3em{\hrulefill}\thinspace}
\providecommand{\MR}{\relax\ifhmode\unskip\space\fi MR }
\providecommand{\MRhref}[2]{%
  \href{http://www.ams.org/mathscinet-getitem?mr=#1}{#2}
}
\providecommand{\href}[2]{#2}

\end{document}